\documentclass[a4paper,11pt]{amsart}

\usepackage{amssymb}
\usepackage[pagebackref,
    ,pdfborder={0 0 0}
    ,urlcolor=black,a4paper,hypertexnames=false]{hyperref}
\hypersetup{pdfauthor={Alberto Casali, Marco Moraschini},pdftitle=Topological volumes of certain complete affine manifolds}

\usepackage{enumerate}
\usepackage{pgf}
\usepackage{tikz}
\usetikzlibrary{decorations.pathmorphing,calc}
\tikzstyle{every picture}+=[remember picture]

\usepackage{tikz-cd}

\newtheorem{thm}{Theorem}[section]
\newtheorem{prop}[thm]{Proposition}

\newtheorem{cor}[thm]{Corollary}
\newtheorem{conj}[thm]{Conjecture}

\newtheorem{question}[thm]{Question}
\newtheorem{setup}[thm]{Setup}

\theoremstyle{remark}
\newtheorem{rem}[thm]{Remark}
\newtheorem{example}[thm]{Example}

\theoremstyle{definition}
\newtheorem{defi}[thm]{Definition}

\usepackage{amsfonts}
\newcommand{\Z}{\mathbb{Z}}

\newcommand{\R}{\mathbb{R}}
\newcommand{\N}{\mathbb{N}}

\newcommand{\Hyp}{\mathbb{H}}

\DeclareMathOperator{\calU}{\mathcal{U}}

\DeclareMathOperator{\cat}{cat}
\DeclareMathOperator{\Am}{Am}
\DeclareMathOperator{\amcat}{\cat_{\Am}}

\DeclareMathOperator{\im}{im}

\DeclareMathOperator{\Aff}{Aff}
\DeclareMathOperator{\G}{\mathcal{G}}
\DeclareMathOperator{\Poly}{\textup{Poly}^{\textup{fg}}}

\DeclareMathOperator{\Isom}{Isom}
\DeclareMathOperator{\GL}{GL}

\DeclareMathOperator{\ess}{ess}

\def\actson{\curvearrowright}

\makeatletter
\newcommand\norm{\bBigg@{0.8}}
\makeatother

\newcommand{\ifsv}[2][norm]{\csname #1l\endcsname\bracevert\!#2\!%
                            \csname #1r\endcsname\bracevert}
\newcommand{\ifsvp}[3][norm]{\csname #1l\endcsname\bracevert\!#2\!%
                            \csname #1r\endcsname\bracevert\!^{#3}}

\def\sv#1{\|#1\|}
\def\stisv#1{\|#1\|_\Z^\infty}

\newcommand{\essn}[2][norm]{\csname #1l\endcsname\vert #2 \csname#1r\endcsname\vert_{1,\ess}}

\usepackage{color}
\usepackage{pdfcolmk}

\author[A.~Casali]{Alberto Casali}
\address{Dipartimento di Matematica \\ Universit\`{a} di Pisa \\ 56127~Pisa}
\email{alberto.casali@phd.unipi.it}

\author[M.~Moraschini]{Marco Moraschini}
\address{Dipartimento di Matematica\\ Universit\`{a} di Bologna\\ 40126~Bologna}
\email{marco.moraschini2@unibo.it}

\keywords{complete affine manifolds, injective Seifert fiberings, amenable category, simplicial volume, minimal volume entropy, integral approximation}
\subjclass[2020]{53A15, 57R19, 53C23}

\title[Topological volumes of certain complete affine manifolds]{Topological volumes of certain complete affine manifolds}
\date{\today.\ \copyright{\ A.~Casali, M.~Moraschini}.}
\begin{document}

\begin{abstract}
  We provide an estimate of the amenable category of oriented closed connected complete affine manifolds whose fundamental group contains an infinite amenable normal subgroup. As an application we show that all such manifolds have zero simplicial volume. This answers a question by L\"uck in the case of complete affine manifolds.

  Our construction also provides the vanishing of stable integral simplicial volume and minimal volume entropy. This means that such manifolds satisfy integral approximation.
\end{abstract}

\maketitle


\section{Introduction}

Understanding the topology of affine manifolds is notoriously hard and there are several conjectures that are still widely open. For instance the celebrated Chern conjecture predicts that all closed affine manifolds have vanishing Euler characteristic. This is known to be true for complete affine manifolds~\cite{kostant1975euler} and for special affine manifolds~\cite{klingler2017chern} (a conjecture by Markus would imply that this class of manifolds is in fact equivalent to the previous one~\cite{markus}). Moreover, Chern conjecture also holds for other specific classes of manifolds including, e.g., surfaces~\cite{benzecri1960varietes}, higher rank irreducible locally symmetric manifolds~\cite{goldman1984radiance}, manifolds that are locally a product of hyperbolic planes~\cite{bucher2011milnor}, complex hyperbolic manifolds~\cite{pieters} and certain aspherical affine manifolds~\cite{BCL}. 

In this paper we restrict our attention to \emph{complete} affine manifolds. 

\noindent
A closed affine manifold~$M$ is characterized by two maps: the (unique up to conjugation) \emph{holonomy representation} $\rho \colon \pi_1(M) \to \Aff(\R^n) \cong \R^n \rtimes \textup{GL}(n, \R)$ and the (unique up to affine transformations) \emph{developing map} $D \colon \widetilde{M} \to \R^n$. An affine manifold is said to be complete if the developing map~$D$ is a diffeomorphism -- in particular, such manifolds are aspherical and the holonomy maps are injective~\cite{goldman2022geometric}.

An open set $U \subset M$ is called \emph{amenable} if the image $ \im(\pi_1(U, x) \to \pi_1(M, x))$ is an amenable group for all $x \in U$.
We denote by $\amcat(M)$ the minimum cardinality of an open amenable cover of $M$. Our main result is the following:

\begin{thm}[Theorem~\ref{thm:main:amcat}]\label{thm:intro:amcat:complete}
    Let $M$ be a closed connected complete affine manifold such that its fundamental group contains an infinite amenable normal subgroup. Then $\amcat(M) \leq \dim(M)$.
\end{thm}

In particular, the previous result remains true if we restrict our attention to the smaller class of open covers $\calU$ such that for all $U \in \calU$ the image $\im(\pi_1(U, x) \to \pi_1(M, x))$ has \emph{polynomial growth} for all $x \in U$.

\medskip

Finding an example of a complete affine manifold that does not satisfy the hypothesis of our theorem is hard. Indeed, the Auslander conjecture predicts that all closed connected affine manifolds have amenable fundamental group. This shows that any closed complete affine manifold whose fundamental group does not contain an infinite amenable normal subgroup would be a counterexample of Auslander conjecture.
On the other hand, Auslander conjecture is known to be true in dimension~$n \leq 6$~\cite{abels2020auslander} and so for closed connected complete affine manifolds of those dimensions we get that the amenable category is equal to $1$.

\subsection{Applications to simplicial volume}

Simplicial volume is a topological invariant of closed manifolds that has been introduced by Gromov~\cite{vbc}. For affine manifolds it provides an upper bound for the Euler characteristic~\cite[Section~13.4]{frigerio}. This led L\"uck~\cite[Question~14.53]{lueckl2} and independently Bucher--Connell--Lafont~\cite[p.~1288]{BCL} to formulate a stronger version of Chern conjecture:

\begin{conj}\label{conj:sv:affine:zero}
Let $M$ be a closed affine manifold. Then the simplicial volume of $M$ vanishes.
\end{conj}

In the case of complete affine manifolds, Bucher, Connell and Lafont proved the following:

\begin{thm}[{\cite[Corollary~1.1]{BCL}}]
Let $M$ be an oriented closed connected complete affine manifold whose holonomy contains a pure translation, then the simplicial volume of $M$ vanishes.
\end{thm}

Recall that the holonomy of a complete affine manifold~$M$ \emph{contains a pure traslation} if the subgroup~$\rho(\pi_1(M)) \cap \R^n < \Aff(\R^n)$ is nontrivial. This shows that the fundamental groups of such manifolds contain an infinite abelian normal subgroup. On the other hand, there exist complete affine manifolds that do not contain pure translations but whose fundamental groups still have an infinite amenable normal subgroup. Such manifolds can be already constructed in dimension~$2$ due to a result by Kuiper~\cite{kuiper1953surfaces} (Example~\ref{ex:kuiper}).

A direct application of Gromov's Vanishing Theorem (Theorem~\ref{thm:gromov:van:thm}) shows that Theorem~\ref{thm:intro:amcat:complete} implies the following:

\begin{cor}[Corollary~\ref{cor:van:sv:stisv}]\label{cor:main:intro}
    Let $M$ be an oriented closed connected complete affine manifold whose fundamental group contains an infinite amenable normal subgroup. Then the simplicial volume of $M$ vanishes.
\end{cor}

Even if our result does not cover all complete affine manifolds, it shows that for complete affine manifolds the following question by L\"uck has a positive answer:

\begin{question}[{\cite[Question~14.39]{lueckl2}}]\label{quest:luck}
    Let $M$ be an oriented closed connected aspherical manifold. Suppose that the fundamental group of $M$ contains an infinite amenable normal subgroup. Is it true that the simplicial volume of $M$ vanishes?
\end{question}

Following a construction by Lee and Raymond~\cite{leeraymond}, in Section~\ref{sec:concluding:rem} we show that L\"uck's question has a positive answer when the infinite amenable normal subgroup is virtually poly-$\Z$ and the quotient group satisfies some natural hypotheses (Theorem~\ref{thm:polyZkernel}).

\subsection{Applications to integral approximation}

Theorem~\ref{thm:intro:amcat:complete} also provides a new family of manifolds satisfying integral approximation~\cite{LMS, loehergodic}:

\begin{cor}[Corollary~\ref{cor:van:sv:stisv}]
    Let $M$ be an oriented closed connected complete affine manifold whose fundamental group contains an infinite amenable normal subgroup. Then \[
    \stisv{M} = \sv{M} = 0.
    \]
\end{cor}

Here $\stisv{\cdot}$ and $\sv{\cdot}$ denote the stable integral simplicial volume and the simplicial volume, respectively (Section~\ref{sec:sv}).

\subsection{Applications to minimal volume entropy}

Minimal volume entropy is another topological invariant of closed manifolds introduced by Gromov~\cite{vbc}. Also in this case (the refined version in terms of polynomial growth of) Theorem~\ref{thm:intro:amcat:complete} together with a vanishing result by Babenko and Sabourau~\cite{bsfibre} implies the following:

\begin{cor}[Corollary~\ref{cor:minent:van}]
    Let $M$ be an oriented closed connected complete affine manifold whose fundamental group contains an infinite amenable normal subgroup. Then the minimal volume entropy of $M$ vanishes.
\end{cor}

\subsection{Strategy of the proof}
Our proof is mainly topological and it uses the theory of injective Seifert fiberings by Lee and Raymond~\cite{leeraymond}: we first deduce an estimate for the (generelized) amenable category of injective Seifert fiber spaces whose typical fiber is amenable (Theorem~\ref{thm:amcat:estimate:SF}). Then via a classical result by Kamishima~\cite{kamishima}, we prove that all complete affine manifolds whose fundamental group contains an infinite amenable normal subgroup admit a structure of injective Seifert fiber spaces whose typical fiber is a torus of positive dimension.
The combination of the previous two results leads to Theorem~\ref{thm:intro:amcat:complete} (Theorem~\ref{thm:main:amcat}).

\subsection{Plan of the paper}
In Section~\ref{sec:sv} we recall the definitions of simplicial volume, stable integral simplicial volume and minimal volume. Moreover, we introduce the notion of amenable category (and its generalizations) and state all the vanishing theorems we need. Section~\ref{sec:affine} is devoted to the definition of affine manifolds and some examples. In Section~\ref{sec:SF} we recall the definition of injective Seifert fiber spaces and prove a local description of such manifolds (Proposition~\ref{prop:local:structure:SF}). Then in Section~\ref{sec:proof} we give the proof of Theorem~\ref{thm:intro:amcat:complete} (Theorem~\ref{thm:main:amcat}). The proofs of the vanishing results of the topological volumes are contained in Subsection~\ref{subsec:vanishing}. Finally, in Section~\ref{sec:concluding:rem} we provide an algebraic criterion for the existence of aspherical manifolds for which Question~\ref{quest:luck} has a positive answer.

\subsection{Acknowledgements}
We are grateful to Chris Connell for his interest in our project and useful conversations around complete affine manifolds. Moreover, we thank Giuseppe Bargagnati for his comments on the first draft of this paper.

M.\ M. was supported by the ERC ``Definable Algebraic Topology" DAT - Grant Agreement no.~101077154. 
This work has been funded by the European Union - NextGenerationEU under the National Recovery and Resilience Plan (PNRR) - Mission 4 Education and research - Component 2 From research to business - Investment 1.1 Notice Prin 2022 -  DD N.~104 del 2/2/2022, from title ``Geometry and topology of manifolds", proposal code 2022NMPLT8 - CUP J53D23003820001.


\section{(Stable integral) simplicial volume, minimal volume entropy and amenable covers}\label{sec:sv}

In this section we recall the definitions of (stable integral) simplicial volume, minimal volume entropy and (generalized) amenable category as well as some basic results.

\subsection{(Stable integral) simplicial volume} \emph{Simplicial volume} is a homotopy invariant of closed manifolds introduced by Gromov~\cite{vbc}. Let $M$ be an $n$-dimensional oriented closed connected manifold. Let $R$ be either $\R$ or $\Z$ and let $c = \sum_{i = 1}^{k} a_i \sigma_i \in C_n(M; R)$ be a singular chain. We define the $\ell^1$-\emph{norm} of $c$ to be $\sv{c}_1 = \sum_{i = 1}^k |a_i|$. Given a class in homology $\alpha \in H_n(M; R)$ we define the $\ell^1$-\emph{semi-norm} as
\[
\sv{\alpha}_1 := \inf \{ \sv{c}_1 \, | \, c \mbox{ is a cycle representing $\alpha$}\}.
\]

\begin{defi}[$R$-simplicial volume]
    The $R$-\emph{simplicial volume} of $M$ is defined as the $\ell^1$-seminorm of the $R$-fundamental class $[M]_R$ of $M$
    \[
    \sv{M}_R := \sv{[M]_R} \in [0, \infty).
    \]
\end{defi}

When $R = \R$ we will call the real simplicial volume simply \emph{simplicial volume} and denote it by $\sv{M}$. It is well known that the simplicial volume of oriented closed conneceted hyperbolic manifolds is positive~\cite{thurston, vbc} and it vanishes for many oriented closed connected manifolds including, e.g., manifolds admitting self-maps $f \colon M \to M$ of degree~$|\textup{deg}(f)|\geq 2$~\cite{vbc, frigerio}, manifolds with amenable fundamental group~\cite{vbc, frigerio} as well as all the manifolds listed below (Example~\ref{ex:IA}).
As already witnessed by the circle, the real and the integral simplicial volumes have different vanishing behaviour (since the integral simplicial volume is always strictly positive by definition~\cite{loeh:isv}). However, the situation becomes more interesting when we consider the following stabilization of the integral simplicial volume:

\begin{defi}[stable integral simplicial volume]
    Let $M$ be an oriented closed connected $n$-manifold. The \emph{stable integral simplicial volume} is defined as follows:
    \[
    \stisv{M} := \inf \Big\{ \frac{\sv{N}_{\Z}}{d} \Big| N \to M \mbox{ $d$-sheeted covering}\Big\} \in [0, +\infty) 
    \]
\end{defi}

A classical question is whether the vanishing of the simplicial volume implies the vanishing of the stable integral simplicial volume:

\begin{question}[integral approximation~{\cite[Question~6.2.2]{loehergodic}}]\label{quest:IA}
    Let $M$ be an oriented closed connected aspherical manifold with residually finite fundamental group such that $\sv{M} = 0$. Is it true that $\stisv{M} = 0$?
\end{question}

\begin{example}\label{ex:IA}
Question~\ref{quest:IA} is known to be true, e.g., in the following cases:
\begin{enumerate}
\item Aspherical manifolds with residually finite amenable fundamental groups;
\item Aspherical manifolds that are graph $3$-manifolds;
\item Smooth aspherical manifolds that admits a regular smooth circle foliation with finite holonomy groups and have residually finite fundamental group;
\item Smooth aspherical manifolds that admit a smooth $S^1$-action without fixed points and have residually finite fundamental group
\item Manifolds that admit an $F$-structure (of possibly zero rank) and have residually finite fundamental group.
\end{enumerate}
\end{example}

\subsection{Minimal volume entropy} \emph{Minimal volume entropy} is a homotopy invariant of closed smooth manifolds introduced by Gromov that describes the asymptotic geometry in terms of the minimal growth rate of balls~\cite{vbc}.

\begin{defi}[minimal volume entropy]
    Let $M$ be an oriented closed connected smooth manifold and let $g$ be a Riemannian metric on $M$. Then the \emph{volume entropy} of $(M, g)$ is defined as
    \[
    \textup{ent}(M, g) := \lim_{R \to +\infty} \frac{1}{R} \log \Big( \textup{vol}(B_R(\widetilde{x}))\Big),
    \]
    where $B_R(\widetilde{x})$ denotes a ball centered at $\widetilde{x}$ in the universal covering of $M$ and the volume is measured by taking the pullback metric $\widetilde{g}$. This definition does not depend on the chosen point~$\widetilde{x}$.

    The \emph{minimal volume entropy} of $M$ is defined as follows
    \[
    \textup{minent}(M) := \inf_{g \in \textup{Riem(M)}} \textup{ent}(M, g) \cdot \textup{vol}(M, g)^{\frac{1}{\dim(M)}} \in [0, +\infty).
    \]
\end{defi}

A classical result by Gromov~\cite[p. 37]{vbc} shows that for all oriented closed connected $n$-manifolds $M$ we have:
\[
c_{\dim(M)} \sv{M} \leq \textup{minent}(M)^{\dim(M)},
\]
where $c_{\dim(M)}$ is a positive constant. On the other hand, it is a widely open question whether the vanishing of the simplicial volume implies the vanishing of the minimal volume entropy~\cite{bsfibre}. This is indeed only known to be true in dimension $2$~\cite{katok1982entropy}, $3$~\cite{pieroni2019minimal} and $4$ provided that the $4$-manifold is geometrizable~\cite{suarez2009minimal}. Babenko and Sabourau have recently provided some new evidence showing that the minimal volume entropy seminorm and the simplicial volume are equivalent~\cite{babenko2023volume}.

\subsection{(Generalized) amenable covers}

One can generalize the classical notion of LS-category~\cite{CLOT} to subsets whose fundamental group have some given algebraic properties~\cite{GGH, CLM, lmfibre}. More precisely, let $\mathcal{G}$ be either the class of amenable groups or the one of finitely generated groups with polynomial growth (i.e.\ finitely generated virtually nilpotent groups by a celebrated result of Gromov~\cite{gromov1981groups}). We will denote the first class of groups by $\Am$ and the latter by $\Poly$.

\begin{defi}[$\G$-category]\label{defi:G-cat}
    Let $\G$ be as above and let $X$ be a topological space. Let $U \subset X$ be a (possibly disconnected) open subset. We say that $U$ is a $\G$-\emph{set} if 
    \[
    \im(\pi_1(U \hookrightarrow X, x)) < \pi_1(X, x) 
    \]
    is a group in $\G$ for every $x \in U$.

    We define the $\G$-\emph{category} of $X$, denoted by $\cat_{\G}(X)$, to be the minimal number~$n \in \N$ such that there exists an open cover $\calU$ of $X$ consisting of $\G$-sets and whose cardinality is $n$. We set $\cat_{\G}(X) = +\infty$ if there is no open cover of finite cardinality consisting of $\G$-sets.

    When $\G = \Am$ the $\Am$-category~$\amcat(X)$ is called \emph{amenable category} and $\Am$-sets are called \emph{amenable sets}.
    \end{defi}

    \begin{rem}
        Since the class $\Poly$ is smaller than $\Am$, we have the inequality $\amcat(X) \leq \cat_{\Poly}(X)$ for all topological spaces~$X$.
    \end{rem}

Gromov's Vanishing Theorem~\cite{vbc, CLM} shows that manifolds with ``small" amenable category have vanishing simplicial volume:

\begin{thm}[Vanishing Theorem for simplicial volume~{\cite[p. 40]{vbc}}]\label{thm:gromov:van:thm}
    Let $M$ be an oriented closed connected manifold such that $\amcat(M) \leq \dim(M)$. Then $\sv{M} = 0$.
\end{thm}

Recently some extensions of the previous theorem have been proved. Indeed, restricting to aspherical manifolds the existence of a small open cover also implies the vanishing of the stable integral simplicial volume:

\begin{thm}[Vanishing Theorem for stable integral simplicial volume~{\cite{LMS}}]\label{thm:van:thm:sisv}
Let $M$ be an oriented closed connected aspherical manifold with residually finite fundamental group such that $\amcat(M) \leq \dim(M)$. Then \[\stisv{M} = \sv{M} = 0.\]  
\end{thm}

In particular, all aspherical manifolds which admit a small amenable cover satisfy integral approximation (Question~\ref{quest:IA}). This leads to the following intriguing open question:

\begin{question}[{\cite[Question~1.3]{LMS}}]
    Does there exist an oriented closed connected aspherical manifold $M$ such that $\sv{M} = 0$ but $\amcat(M) = \dim(M)+1$?
\end{question}

Finally, if we restrict to $\cat_{\Poly}(M)$ Babenko and Sabourau proved the following (compare with~\cite[Example~5.10]{lmfibre}):

\begin{thm}[Vanishing Theorem for minimal volume entropy~{\cite[Corollary~1.4 and Proposition~2.2]{bsfibre}}]\label{thm:BS:vanishing:minent}
    Let $M$ be an oriented closed connected smooth manifold such that $\cat_{\Poly}(M) \leq \dim(M)$. Then $\textup{minent}(M) = 0$.
\end{thm}

\section{Affine manifolds}\label{sec:affine}

In this section we briefly recall some basic definitions around the theory of (complete) affine manifolds. We refer the reader to the nice book by Goldman~\cite{goldman2022geometric} for a thorough discussion.

We recall the definition of $(G,X)$\emph{-structure} on a smooth manifold.
\begin{defi}[$(G,X)$-structures and $(G,X)$-manifolds]
    Let $X$ be a smooth $n$-manifold and let $G$ be a group of diffeomorphisms of $X$ such that two diffeomorphisms $g_1, g_2 \in G$ are equal on some open subset of $X$ if and only if $g_1 = g_2$. Let $M$ be a smooth $n$-manifold. A $(G,X)$\emph{-structure} on $M$ is the datum of an open atlas $\{(U_i,\phi_i)\}_{i\in I}$ such that
    \begin{itemize}
        \item For every $i \in I$ the chart $\phi_i\colon U_i\to X$ is a smooth embedding;
        \item For every $(i, j) \in I \times I$ the composition $\phi_j\circ \phi_i^{-1}\colon \phi_i(U_i\cap U_j)\to \phi_j(U_i\cap U_j)$ coincides on each connected component of $\phi_i(U_i\cap U_j)$ with a (necessarily unique) diffeomorphism $g \in G$.
    \end{itemize}

    A $(G,X)$\emph{-manifold} is a smooth manifold that supports a $(G,X)$-structure.
\end{defi}
\begin{example}
    Many geometric properties can be expressed in the language of $(G,X)$-structures, for example:
    \begin{enumerate}
        \item $M$ is a hyperbolic manifold $\Leftrightarrow$ $M$ is a $(\Isom(\Hyp^n),\Hyp^n)$-manifold;
        \item $M$ is a flat manifold $\Leftrightarrow$ $M$ is a   $(\Isom(\R^n), \R^n)$-manifold.
     \end{enumerate}
\end{example}
A $(G,X)$-structure on a manifold $M$ is completely described by a group homomorphism $\rho\colon \pi_1(M)\to G$, called the \emph{holonomy representation}, and a $\rho$-equivariant local diffeomorphism $D\colon \widetilde{M}\to X$, called the \emph{developing map}. The holonomy representation (resp. the developing map) is unique up to conjugation by (resp. composition with)  an element of $G$ \cite{goldman2022geometric}. A $(G,X)$-structure (resp. a $(G,X)$-manifold) is \emph{complete} if the developing map is a covering map.

\begin{defi}[(complete) affine manifolds]
Let $\Aff(\R^n)\cong \R^n\rtimes \GL(n,\R)$ be the group of affine transformations of $\R^n$.
An \emph{affine manifold} is a manifold that support a $(\Aff(\R^n),\R^n)$-structure. 

Since $\R^n$ is simply connected, an affine manifold $M$ is \emph{complete} if and only if the developing map $D\colon \widetilde{M}\to \R^n$ is a diffeomorphism. 
\end{defi}
Clearly, the holonomy representation $\rho$ of a complete affine manifold is injective~\cite[p.~184]{goldman2022geometric}. Moreover, an affine manifold $M$ is complete if and only if it is diffeomorphic to a quotient of $\R^n$ by a discrete subgroup of $\Aff(\R^n)$ acting on $\R^n$ freely and properly~\cite[Section~8.1.2]{goldman2022geometric}.

Quite surprisingly, already in dimension $2$ we can find complete affine manifolds with abelian fundamental group whose \emph{holonomy} (i.e.\ the image of the holonomy representation) does \emph{not} contain any non-trivial pure translation. This example is due to Kuiper~\cite{kuiper1953surfaces} (see also Goldman's book~\cite[pp.~190--192]{goldman2022geometric}):

\begin{example}\label{ex:kuiper} For every $(s,t)\in \R^2$, we define the affine transformation
\[\phi_{s,t} \colon \R^2 \to \R^2\]
as follows
\[ \phi_{s,t}(x,y)= \left(x+ty+s+\frac{t^2}{2}, y+t\right).\]
    The resulting map
    \begin{align*}
        \Phi\colon \R^2&\to \textup{Aff}(\R^2) \\
        (s, t) &\mapsto \phi_{s, t}
    \end{align*}
is a group homomorphism and $\Phi(\R^2)$ acts simply transitively on $\R^2$. For every lattice $\Lambda < \R^2$ the quotient $\R^2/\Phi(\Lambda)$ is diffeomorphic to the $2$-torus $\mathbb{T}^2$. This defines a complete affine structure on $\mathbb{T}^2$ with holonomy $\Phi(\Lambda)$. By construction $\phi_{s,t}$ is a pure translation if and only if $t=0$. Hence the desired example can be obtained by choosing a lattice $\Lambda$ such that \[\Lambda \cap \big(\R\times\{0\}\big)=\{(0,0)\}.\]
\end{example}
Fundamental groups of \emph{closed} complete affine manifolds (or more generally finitely generated groups~$\Gamma \subset \Aff(\R^n)$ acting properly and cocompactly on $\R^n$) are \emph{indecomposable}~\cite[Lemma 1.2]{kamishima}, i.e.\ it is not possible to write all the elements of such groups in the fixed form
\begin{equation*}
    \begin{pmatrix}
        A_1 & 0& 0\\
        0 & A_2 & b\\
        0 & 0 & 1\\
    \end{pmatrix}
\end{equation*}
when viewed as elements of $\GL(n+1,\R).$
\section{Injective Seifert fiberings}\label{sec:SF}

In this section we recall the definitions of \emph{injective Seifert fiberings} and \emph{injective Seifert fiber spaces} due to Lee and Raymond~\cite{leeraymond}.
Examples of such spaces include, e.g., compact nilmanifolds as well as certain complete affine manifolds (Example~\ref{ex:inj:SF}).
Moreover, we also prove a useful local description of injective Seifert fiber spaces (Proposition~\ref{prop:local:structure:SF}).

\begin{setup}\label{setup:SF}
We consider the following setup:
\begin{itemize}
    \item Let $G$ be a connected Lie group;
    \item Let $W$ be a contractible manifold;
    \item Suppose that $G$ acts on $G\times W$ via left-multiplication on the first factor and trivially on the second factor;
    \item Let $\ell\colon G\to \textup{Homeo}(G\times W)$ be the group homomorphism associated to $G \actson G \times W$;
    \item Let $\pi\subset \textup{Homeo}(G\times W)$ be a \emph{discrete} group acting freely and properly on $G\times W$.
\end{itemize}

Moreover, suppose that the following three conditions hold:
\begin{enumerate}
    \item $\ell(G)$ is normalized by $\pi$ in $\textup{Homeo}(G\times W)$;
    \item $\Gamma:= \pi \cap \ell(G)$ is normal in $\pi$ and $\ell^{-1}(\Gamma)$ is discrete in $G$;
    \item The induced action of the discrete group $Q:=\pi/\Gamma$ on $W$ is proper.
\end{enumerate}
\end{setup}
In the situation of Setup~\ref{setup:SF} the projection map $p_W \colon G\times W\to W$ is $\pi$-equivariant, where $\pi$ acts on $W$ via $Q$. Hence $p_W$ induces a continuous map
\[ p\colon (G\times W)/\pi \to W/Q.\]

\begin{defi}[injective Seifert fibering]
  In the situation of Setup~\ref{setup:SF} we say that the map $p$ is an \emph{injective Seifert fibering} modeled on the product bundle $G\times W \to W$. Moreover, $(G\times W)/\pi$ is called \emph{injective Seifert fiber space}, $B:= W/Q$ is the \emph{base} of the injective Seifert fibering $p$ and the manifold $G \slash \Gamma$ is the \emph{typical fiber} of $p$.
\end{defi}

The definition of injective Seifert fibering gives the following commutative diagram
\begin{center}
\begin{tikzcd}
G\times W \arrow[d, "/\Gamma" '] \arrow[rr, "p_W"]     &  & W \arrow[d,  "\textup{Id}_W"] &  & \text{principal $G$-bundle}       \\
G/\Gamma\times W \arrow[ d, "/Q"'] \arrow[rr, "q_W"] &  & W \arrow[ d, "/Q" ']                &  & \text{intermediate fiber bundle}  \\
(G\times W)/\pi \arrow[rr, "p"]                                    &  & B                                &  & \text{injective Seifert fibering}

\end{tikzcd}
\end{center}
where $p_W$ and $q_W$ are the projections onto the second factors.
The commutativity of the lower square is a consequence of the $Q$-equivariance of the projection map $q_W$. Since for convenience we have required in Setup~\ref{setup:SF} that the action $\pi \actson G \times W$ is free, the resulting quotient $(G\times W)/\pi$ is in fact a \emph{manifold}. For this reason we will denote it by $M := (G\times W)/\pi$. Moreover, the two vertical maps on the left in the diagram above are \emph{regular coverings}. 
\begin{defi}[smooth injective Seifert fibering]
    In the situation of Setup~\ref{setup:SF} we say that an injective Seifert fibering is \emph{smooth} if $W$ is a smooth manifold and $\pi$ acts smoothly on $G\times W$. 
\end{defi}

In the case of smooth injective Seifert fiberings the injective Seifert fiber space is a smooth manifold. Moreover, the induced action of $Q$ on $W$ is also smooth, hence the base space $B$ is a smooth orbifold~\cite[Prop. 1.5.1]{caramello}. 

\begin{example}\label{ex:inj:SF}
We collect here some examples of injective Seifert fiberings:
\begin{enumerate}
    \item Let $\Gamma$ be a discrete subgroup of $G$, then the constant map $G/\Gamma \to \{\ast\}$ is an injective Seifert fibering modeled on the product bundle $G\times \{\ast\} \to \{\ast\}$. 
    \item Let $G$ be a simply connected nilpotent Lie group and let $C$ be a maximal compact subgroup of the group of Lie automorphisms~$\textup{Aut}(G)$. An \emph{almost-Bieberbach group}~$\pi$ modeled on $G$ is a torsion-free discrete cocompact subgroup of $G \rtimes C$. The Generalized First Bieberbach Theorem~\cite[Theorem~1]{auslandernilpotent} states that the group of pure translations $\Gamma=\pi\cap G\trianglelefteq \pi$ is a lattice of $G$ and $Q=\pi/\Gamma$ finite. Then the constant map $G \slash \pi \to \{\ast\}$ is an injective Seifert fibering modeled on the product bundle $G \times \{\ast\} \to \{\ast\}$. In particular, the injective Seifert fiber space~$G/\pi$ is an \emph{infra-nilmanifold}.
    \item Let $\mathbb{T}^k$ denote the $k$-dimensional torus. Let $M$ be a connected $n$-manifold that admits an effective \emph{injective} $\mathbb{T}^k$-action for some $k<n$. Recall that a torus action is  \emph{injective} if for every $x\in M$ the evaluation map $\textup{ev}_x\colon \mathbb{T}^k\to M$ defined as $\textup{ev}_x(t)=t\cdot x$ induces an injective homomorphism between the fundamental groups.
    For instance every effective torus action on closed and aspherical manifolds is injective~\cite[Corollary~3.1.12]{leeraymond}. In the case of injective torus action, one can show that the image $\textup{ev}_x^*(\pi_1(\mathbb{T}^k))\cong \Z^k$ is a central subgroup of $\pi_1(M)$~\cite[Theorem 2.4.2]{leeraymond}. Moreover, the universal covering $\widetilde{M}$ is homeomorphic to $\R^k \times W$ for some simply connected manifold $W$, the group~$\textup{ev}_x^*(\pi_1(\mathbb{T}^k))\cong \Z^k$  acts on  $\R^k \times W$ by translations on the first factor and trivially on the second factor, and the quotient group~$Q=\pi_1(M)/\textup{ev}_x^*(\pi_1(\mathbb{T}^k))$ acts properly on $W$~\cite[Section 4.5.3]{leeraymond}. This shows that the induced map $M\to W/Q$ is an injective Seifert fibering modeled on product bundle $\R^k\times W$.
   \item Let $M$ be a complete affine manifold with indecomposable fundamental group, such that $\pi_1(M)$ admits an infinite abelian normal subgroup. Then $M$ is a smooth injective Seifert fiber space with typical fiber  $\mathbb{T}^k$ for some $k>0$~\cite[Proposition~1.3]{kamishima}.
\end{enumerate}
    
\end{example}
The examples above showcase that injective Seifert fiberings aim to generalize the splitting property of the universal covering of injective torus actions: when $\pi$  \emph{centralizes} $\ell(G)$, the $G$-action on $G\times W$ still induces a $G/\Gamma$-action on $X$. On the other hand, injective Seifert fiberings also extend the classical notion of Seifert fibered $3$-manifold, as witnessed by the following proposition.

\medskip

\begin{prop}[local structure of injective Seifert fiberings]~\label{prop:local:structure:SF}
    In the situation of Setup~\ref{setup:SF}, let $p \colon M \to B$ be a smooth injective Seifert fibering modeled on the product bundle $G \times W \to W$ with typical fiber~$G \slash \Gamma$. Then for every $b \in B$ there exists an open neighbourhood $V_b$ of $b$ such that $p^{-1}(V_b)$ admits a regular covering of finite degree diffeomorphic to $G/\Gamma\times\R^{\dim(W)}$.
\end{prop}

A fundamental ingredient in the proof of Proposition~\ref{prop:local:structure:SF} is the following version of the Slice Theorem for discrete groups that we reprove for convenience of the reader:

\begin{thm}[Slice Theorem for discrete groups]\label{thm:slice}
    Let $\Gamma$ be a discrete group acting properly and smoothly on a smooth manifold $M$. Let $x \in M$ be a point and let $\Gamma_x$ denote the stabilizer of $\Gamma$ at $x$. Then 
    \begin{enumerate}
        \item There exists a $\Gamma_x$-invariant open neighbourhood~$V$ of $x$ that is diffeomorphic to $\R^{\dim(M)}$ and has the following two additional properties: $\{g \in \Gamma \, | \, gV \cap V \neq \emptyset \} = \Gamma_x$ and $g V \cap g' V \neq \emptyset$ if and only if $g$ and $g'$ belong to the same $\Gamma_x$-coset;

        \item The open set $U := \Gamma \cdot V$ is a $\Gamma$-invariant open neighbourhood of $\Gamma \cdot x$ diffeomorphic to $\Gamma \slash \Gamma_x \times \R^{\dim(M)}$.
    \end{enumerate}
\end{thm}
\begin{proof}
\emph{Ad~1)} Since $\Gamma$ acts properly on $M$, it is well known that each $x\in M$ admits a $\Gamma_x$-invariant open neighbourhood $S$ such that $\{g \in \Gamma \, | \, g S \cap S \neq \emptyset \} = \Gamma_x$ and $g S \cap g' S \neq \emptyset$ if and only if $g$ and $g'$ belong to the same $\Gamma_x$-coset \cite[Proposition~1.7.4]{leeraymond}. We claim that there exists  an open neighbourhood $V$ with the same properties that is also diffeomorphic to $\R^{\dim(M)}$.
     Indeed, since $\Gamma_x$ is finite we can choose a $\Gamma_x$-invariant Riemannian metric on $M$ and then consider the exponential map $\textup{exp} \colon T_xM \to M$ with respect to that metric. By construction $\textup{exp}$ is $\Gamma_x$-equivariant and its differential at the origin is invertible. Hence it provides a diffeomorphism between a small open ball $B_\varepsilon(0) \subset T_xM$ and its image $\textup{exp}(B_\varepsilon(0))$. Up to taking a smaller $\varepsilon >0$, we can also assume that $\textup{exp}(B_\varepsilon(0))\subset S$. The set~$V := \textup{exp}(B_\varepsilon(0))$ provides the desired open neighbourhood of $x$. Indeed, since $\exp$ is $\Gamma_x$-equivariant and the chosen Riemannian metric on $M$ is $\Gamma_x$-invariant, the set $V$ is $\Gamma_x$-invariant. Hence the inclusions $\{x\} \subset V\subset S$ readily imply that $\{g \in \Gamma \, | \, g V \cap V \neq \emptyset \} = \Gamma_x$ and that $g V \cap g' V \neq \emptyset$ if and only if $g$ and $g'$ belong to the same $\Gamma_x$-coset. In fact, the $\Gamma_x$-invariance of $V$ shows more: $g V \cap g' V \neq \emptyset$ if and only if $g V = g' V$.

     \emph{Ad~2)} Consider $U:=\Gamma \cdot V$. By construction $U$ is a $\Gamma$-invariant open neighbourhood of the orbit $\Gamma\cdot x$. Moreover, since $g V \cap g' V \neq \emptyset$ if and only if $g$ and $g'$ belong to the same $\Gamma_x$-coset, we have that $U$ is the disjoint union of copies of $V$ indexed by $\Gamma/\Gamma_x$. In other words, we have shown that
     \[ U\cong \Gamma/\Gamma_x\times V \cong \Gamma/\Gamma_x\times \R^{\dim(M)}\] where $\Gamma/\Gamma_x$ is endowed with the discrete topology.
\end{proof} 

We are now ready to prove Proposition~\ref{prop:local:structure:SF}:

\begin{proof}[Proof of Proposition~\ref{prop:local:structure:SF}]
Consider the following commutative diagram:
\begin{center}
\begin{tikzcd}
G/\Gamma\times W \arrow[ d, "p_M", "/Q"'] \arrow[rr, "q_W"] &  & W \arrow[ d, "p_B", "/Q" ']        \\
M \arrow[rr, "p"]                                    &  & B,                              
\end{tikzcd}
\end{center}
where $q_W$ is the projection onto the second factor. Fix a point $b\in B$ and choose a point $w\in p_B^{-1}(b)\subset W$.
According to Setup~\ref{setup:SF} the discrete group $Q:=\pi/\Gamma$ acts properly on $W$. Let $Q_w$ denote the \emph{finite} stabilizer of the point $w$.
    By Theorem~\ref{thm:slice}.(1) there exists a $Q_w$-invariant open neighbourhood~$U_w$ of $w$ diffeomorphic to $\R^{\dim(W)}$ and by Theorem~\ref{thm:slice}.(2) a $Q$-invariant open neighbourhood~$U:= Q \cdot U_w$ of the orbit $Q \cdot w$ diffeomorphic to $Q/Q_w \times \R^{\dim(W)}$. 
    Then the set $V_b := p_B(U_w)$ is an open neighbourhood of $b$, since the preimage $p_B^{-1}(V_b)=Q \cdot U_w= U$ is open. Moreover, the $Q$-equivariance of $q_W$ implies
\[
q_W^{-1}(p_B^{-1}(V_b))= G/\Gamma\times U = \coprod\limits_{qQ_w \in Q/Q_w} q\cdot (G/\Gamma\times U_w).
\]
By the commutativity of the diagram we have that \[p_M(q_W^{-1}(p_B^{-1}(V_b)))=p^{-1}(V_b).\]
Since $Q$ transitively permutes the connected components in $q_W^{-1}(p_B^{-1}(V_b))$, we have that $p_M$ restricts to a surjective covering map 
\[\psi \colon G/\Gamma\times U_w\to p^{-1}(V_b).\]
Note that for every $(g,v)\in G/\Gamma \times U_w$ we have that 
\[
q\cdot (g,v)\in G/\Gamma \times U_w \Leftrightarrow q\cdot v \in U_w \Leftrightarrow q\in Q_w.
\]
Since $Q_w$ acts freely by deck transformations, this readily implies that for every $x\in p^{-1}(V_b)$
\[ |\psi^{-1}(x)|=|p_M^{-1}(y) \cap (G/\Gamma \times U_w)|=|Q_w| < +\infty.\]
This means that $Q_w$ transitively permutes the points in $\psi^{-1}(x)$ for every $x\in p^{-1}(V_b)$. Hence the map $\psi$ is a regular finite covering with automorphism group $\textup{Aut}(\psi)\cong Q_w$.
\end{proof}
The previous result explains why the manifold $G/\Gamma$ is called \emph{typical} fiber:
when $Q$ acts effectively, the set of points with trivial stabilizer
\[B_\text{reg}=\{b\in B | \ Q_w=1 \ \forall w\in p_B^{-1}(b)\}\]
is open and dense \cite[p. 12]{caramello}.
\section{Proof of the main theorem}\label{sec:proof}

In this section we prove Theorem~\ref{thm:intro:amcat:complete}:

\begin{thm}[$\G$-category of certain complete affine manifolds]\label{thm:main:amcat}
   Let $M$ be a closed connected complete affine manifold such that its fundamental group contains an infinite amenable normal subgroup. Then \[\amcat(M) \leq \cat_{\Poly}(M) \leq \dim(M).\]
\end{thm}

In order to deduce the previous theorem from Proposition~\ref{prop:local:structure:SF} we need the following result of independent interest:

\begin{thm}[$\G$-category of injective Seifert fiber spaces]\label{thm:amcat:estimate:SF}
In the situation of Setup~\ref{setup:SF}, let $M$ be a closed connected smooth manifolds and let $p\colon M \to B$ be a smooth injective Seifert fibering with typical fiber $G \slash \Gamma$. 
Suppose that $\pi_1(G \slash \Gamma)$ is amenable. Then \[\amcat(M) \leq \dim(M) - \dim(G \slash \Gamma) + 1.\]
Moreover, if $\pi_1(G \slash \Gamma)$ is virtually nilpotent, we also have
\[
\cat_{\Poly}(M) \leq \dim(M) - \dim(G \slash \Gamma) + 1.
\]
\end{thm}

We use the following notation: given two covers $\mathcal{U}$ and $\mathcal{V}$ of $M$, we say that $\mathcal{U}$ is \emph{subordinate} to $\mathcal{V}$ if 
for every $U_i \in \mathcal{U}$ each path-connected component $U_i^a$ of $U_i$ is entirely contained in some $V_j^a \in \mathcal{V}$.

\begin{proof}
Let $k:=\dim(B)=\dim(M)-\dim(G\slash\Gamma)$. Since $M$ is closed and the map $p$ is surjective, $B$ must be compact. Hence by Proposition~\ref{prop:local:structure:SF} there exists a finite open cover $\mathcal{V} := \{V_{1}, \cdots, V_{s}\}$ of $B$ such that each preimage $p^{-1}(V_{j})$ admits a regular covering of finite degree diffeomorphic to $G/\Gamma\times\R^{\dim(W)}$ (for sake of notational simplicity we do not stress the dependence of the sets $V_j$ on the points $b_j \in B$). By construction $p^{-1}\mathcal{V} := \{p^{-1}(V_{j})\}_{j = 1}^s$ is an open cover of $M$.
We want to construction now a new open cover of $M$ with cardinality $k+1$ that is subordinate to $p^{-1}\mathcal{V}$.

To this end we define a new open cover~$\mathcal{U}$ of $B$ of cardinality $k+1$ that is subordinate to $\mathcal{V}$. Since $B$ is a smooth compact orbifold, $B$ is a triangulable compact space~\cite[Theorem 2.4.1]{caramello}. In particular, there exists a sufficiently small triangulation~$T$ of $B$ such that each open star of the vertices is entirely contained in some set~$V_{j}$ of the cover~$\mathcal{V}$. The desired open cover~$\mathcal{U} := \{U_0, \cdots, U_k\}$ of~$B$ is defined as follows: each~$U_i$ consists of the disjoint union of the open stars of the vertices of the barycentric subdivision of~$T$ corresponding to the $i$-simplices in $T$.

By pulling back $\mathcal{U}$ along $p$ we obtain an open cover \[p^{-1}\mathcal{U} := \{p^{-1}(U_0),\dots, p^{-1}(U_{k}) \}\] of $M$ of cardinality $k+1$ that is subordinate to $p^{-1}\mathcal{V}$.

Fix an index $i \in \{0,\dots, n\}$ and a basepoint $x \in p^{-1}(U_i)$.
Since $p^{-1}\mathcal{U}$ is subordinate to $p^{-1}\mathcal{V}$, there exists an index ${j_0} \in \{1,\dots, s\}$ such that the path connected component of $p^{-1}(U_i)$ containing $x$ is contained in $p^{-1}(V_{{j_0}})$. Then the homomorphism induced by the inclusion 
$\pi_1(p^{-1}(U_i) \hookrightarrow M, x)$ 
factors through the group~$\pi_1(p^{-1}(V_{j_0}, x))$.

Since both the classes of amenable groups~\cite[Proposition~3.4]{frigerio} and of finitely generated virtually nilpotent groups~\cite[Proposition~4.1]{wolf} are closed under taking subgroups and quotients, we are left to show the following:
\begin{enumerate}
    \item[(\emph{1})] If $\pi_1(G \slash \Gamma)$ is amenable, then also the groups~$p^{-1}(V_{j})$ are amenable for all $j = 0, \cdots, s$, and 
    \item[(\emph{2})] If $\pi_1(G \slash \Gamma)$ is virtually nilpotent, then the groups~$p^{-1}(V_{j})$ are finitely generated virtually nilpotent groups for all $j = 0, \cdots, s$.
\end{enumerate}
\emph{Ad~1)} Since each $p^{-1}(V_{j})$ admits a regular covering of finite degree diffeomorphic to $G/\Gamma\times\R^{\dim(W)}$, we have that its fundamental group is amenable (because virtually amenable groups are amenable). This proves that if the group $\pi_1(G \slash \Gamma)$ is amenable, then we have 
\[
\amcat(M) \leq k+1 = \dim(M) - \dim(G \slash \Gamma) + 1.
\]
\emph{Ad~2)} First observe that $G \slash \Gamma$ is compact since finitely covers the preimage $p^{-1}(b)$ for some $b \in B$ inside the compact manifold~$M$. Hence $\pi_1(G \slash \Gamma)$ is finitely generated. Moreover, it is well known that $\pi_1(G\slash\Gamma)$ contains a finite-index nilpotent subgroup that is \emph{characteristic}~\cite[p. 355]{makarenkoshumyatski}. As before the fundamental group of each~$p^{-1}(V_{j})$ is a finite extension of $\pi_1(G \slash \Gamma)$ and so it is a finitely generated virtually nilpotent group. This proves that if the group $\pi_1(G \slash \Gamma)$ is virtually nilpotent, then we have 
\[
\cat_{\Poly}(M) \leq k+1 = \dim(M) - \dim(G \slash \Gamma) + 1. \qedhere
\]
\end{proof}

We are now ready to prove Theorem~\ref{thm:main:amcat}:

\begin{proof}[Proof of Theorem~\ref{thm:main:amcat}]
    Let $M$ be a closed connected complete affine manifold and suppose that $S \trianglelefteq \pi_1(M)$ is an infinite amenable normal subgroup. By assumption of completeness, we know that the holonomy map is injective and so we have the following inclusions \[S \leq \pi_1(M) \leq \Aff(\R^n) \leq GL(n+1; \R).\] This shows that $S$ is an amenable linear group over a field of characteristic~$0$. Thus by the Tits alternative~\cite[Theorem 1]{tits1972free} we have that $S$ is a virtually solvable group. Then it is known that $S$ contains a \emph{characteristic} solvable group of finite index $S'\trianglelefteq_{\textup{char}} S$ \cite[p. 355]{makarenkoshumyatski}. Let $A$ be the smallest non-trivial group in the derived series of $S'$. Note that $A$ is characteristic in $S'$, hence it is characteristic in $S$ and normal in $\pi_1(M)$. Moreover, $A$ is abelian by construction.
     Since $M$ is aspherical, we also have that $A$ must be torsion free. Hence it follows that $\pi_1(M)$ contains an infinite abelian normal subgroup. By Example~\ref{ex:inj:SF}.(4), we can conclude that $M$ is a smooth injective Seifert fiber space whose typical fiber is a torus of positive dimension. By Theorem~\ref{thm:amcat:estimate:SF} the claim follows.
\end{proof}

\subsection{Vanishing of topological volumes}\label{subsec:vanishing}

The vanishing of simplicial volume, stable integral simplicial volume and minimal volume entropy is a straightforward application of Theorem~\ref{thm:main:amcat}. Indeed, since closed complete affine manifolds have residually finite fundamental group~\cite{malcev}, combining Theorem~\ref{thm:main:amcat} with Gromov's Vanishing Theorem and its generalization (Theorems~\ref{thm:gromov:van:thm} and Theorem~\ref{thm:van:thm:sisv}) we have:
\begin{cor}[integral approximation for certain complete affine manifolds]\label{cor:van:sv:stisv}
    Let $M$ be an oriented closed connected complete affine manifold such that its fundamental group contains an infinite amenable normal subgroup. Then
    \[
    0 =\sv{M} = \stisv{M}.
    \]
\end{cor}

Similarly, combining the vanishing result for the minimal volume entropy by Babenko and Sabourau (Theorem~\ref{thm:BS:vanishing:minent}) with the estimate for $\cat_{\Poly}(M)$ proved in Theorem~\ref{thm:main:amcat} we get the following:
\begin{cor}[minimal volume entropy of certain complete affine manifolds]\label{cor:minent:van}
Let $M$ be an oriented closed connected complete affine manifold such that its fundamental group contains an infinite amenable normal subgroup. Then $\textup{minent}(M) = 0$.
\end{cor}

\section{Concluding remarks}\label{sec:concluding:rem}
While at the time of writing this paper Question \ref{quest:luck} remains widely open in full generality, some partial results have been achieved when the fundamental group of $M$, its amenable normal subgroup, and/or the resulting quotient satisfy additional hypotheses. Indeed a positive answer to Question~\ref{quest:luck} is known, e.g., in the following cases: let $M$ be an orientable closed connected aspherical manifold and let $\pi\cong\pi_1(M)$ be its fundamental group. Suppose that there exists a non-trivial amenable normal subgroup~$\Gamma \trianglelefteq \pi$. Then the simplicial volume of $M$ vanishes if one of the following additional conditions is satisfied
\begin{enumerate}
    \item The map induced in top-dimensional cohomology by the quotient $H^{\dim(M)}(\pi/\Gamma;\R)\to H^{\dim(M)}(\pi;\R)$ is zero~\cite[Lemma 2.1]{BCL};
    \item If $\Gamma = \mathcal{Z}(\pi)$ is the center of~$\pi$, $\pi$ is infinite-index presentable by products and $\pi/\mathcal{Z}(\pi)$ is not presentable by products (both terminologies are due to Neofytidis)~\cite[Corollary 1.5]{neofytidis}.
\end{enumerate}
Using a purely algebraic criterion for the existence of smooth injective Seifert fiber spaces by Lee and Raymond~\cite[Theorem~11.1.4]{leeraymond}, we obtain the following:
\begin{thm}[algebraic criterion for the existence of aspherical manifolds as in L\"uck's question]\label{thm:polyZkernel}
   Let $\Gamma$ be a virtually poly-$\Z$ group and let $Q$ be a discrete group acting properly, smoothly, and cocompactly on some smooth contractible manifold $W$. Suppose that $\pi$ is a discrete torsion-free group fitting in the following short exact sequence
   \[
   1 \to \Gamma \to \pi \to Q \to 1.
   \]
   Then there exists a smooth aspherical manifold $M$ with $\pi_1(M)\cong \pi$ such that $\cat_{\Poly}(M) \leq \dim(M)$ and
    \begin{equation*}
        \sv M= \stisv{M}= \textup{minent}(M) = 0.
    \end{equation*}
\end{thm}
Recall that a  group $\Gamma$ is \emph{poly-}$\Z$  if it admits a subnormal series 
\[
1=\Gamma_0\trianglelefteq\Gamma_1\trianglelefteq \dots\trianglelefteq\Gamma_k\cong \Gamma
\]
such that $\Gamma_{i+1}/\Gamma_i \cong \Z$ for all $i\in\{0,\dots, k-1\}$.
\begin{proof}
    We follow the proof by Lee and Raymond~\cite[Theorem~11.1.4]{leeraymond}. Let $\Delta$ be the maximal nilpotent characteristic subgroup of $\Gamma$~\cite[Subsection~8.4.6]{leeraymond}. Since $\Gamma$ is virtually poly-$\Z$ and $\pi$ is torsion-free, then $\Delta$ is a torsion-free finitely generated nilpotent group and so by a well-known theorem by Mal'cev
    it is a cocompact lattice of a unique simply connected nilpotent Lie group~$N$~\cite{mal1949class}. Moreover, by construction one can show that  $\Gamma \slash \Delta$ is virtually $\mathbb{Z}^s$ for some $s \geq 0$. Since $\Delta$ is characteristic in $\Gamma$, we obtain the following short exact sequence of groups
    \[
    1 \to \Delta \to \pi \to \pi \slash \Delta \to 1.
    \]
    Then in their proof Lee and Raymond show that there exists a smooth aspherical manifold~$M$ with fundamental group isomorphic to~$\pi$ with the following property: $M$ is a smooth injective Seifert fiber space with typical fiber $N \slash \Delta$ and base $(\R^s \times W) \slash (\pi \slash \Delta)$.
    The result now follows from the combination of Theorem~\ref{thm:amcat:estimate:SF} with the vanishing results about the stable integral simplicial volume (Theorems~\ref{thm:gromov:van:thm} and~\ref{thm:van:thm:sisv}) and the minimal volume entropy (Theorem~\ref{thm:BS:vanishing:minent}).
\end{proof}

One may wonder whether Theorem~\ref{thm:polyZkernel} provides a complete answer to Question~\ref{quest:luck} in the case of poly-$\mathbb{Z}$ normal groups. However, this reduces to the following variation of a question by Lee and Raymond~\cite[Question~11.1.9.(i)]{leeraymond}:

\begin{question}
Let $M$ be a closed aspherical manifold and let $\Gamma$ be an infinite poly-$\mathbb{Z}$ normal subgroup of $\pi_1(M)$. Let $Q := \pi_1(M) \slash \Gamma$. Does there exists a contractible manifold $W$ on which $Q$ acts properly with compact quotient?
\end{question}

On the other hand, one may also wonder whether Theorem~\ref{thm:polyZkernel} is a consequence of the result by Bucher--Connell--Lafont above~\cite[Lemma~2.1]{BCL}. Namely, the question reduces to determine whether the real cohomological dimension of $Q$ is smaller than the real cohomological dimension of $\pi$. However, according to Lee and Raymond~\cite[Question~11.1.9.(ii)]{leeraymond} also this question seems to be widely open:

\begin{question}
    Let $Q$ be a discrete group that acts properly and cocompactly on a contractible manifold $W$. Does $Q$ contain a torsion-free subgroup of finite index?
\end{question}

\bibliographystyle{alpha}
\bibliography{svbib}

\end{document}